\documentclass[10pt,reqno]{amsart}

\usepackage{a4wide}
\usepackage{amssymb}
\usepackage{dsfont} 	
\usepackage[all]{xy}
\usepackage{todonotes}

\newtheorem{proposition}{Proposition}
  \newtheorem{theorem}[proposition]{Theorem}
  
  \newtheorem{corollary}[proposition]{Corollary}
\theoremstyle{definition}
  \newtheorem{definition}[proposition]{Definition}
  \newtheorem{remark}[proposition]{Remark}
  
\numberwithin{equation}{section}
\numberwithin{proposition}{section}

\DeclareMathOperator{\Mor}{Mor}
\DeclareMathOperator{\M}{\mathsf{M}}
\DeclareMathOperator{\C}{C}

\newcommand{\cst}{\ifmmode\mathrm{C}^*\else{$\mathrm{C}^*$}\fi}
\newcommand{\ph}{\varphi}
\newcommand{\tens}{\otimes}
\newcommand{\id}{\mathrm{id}}
\newcommand{\comp}{\circ}
\newcommand{\qqquad}{\quad\qquad}
\newcommand{\I}{\mathds{1}}

\newcommand{\CC}{\mathbb{C}}
\renewcommand{\SS}{\mathbb{S}}
\newcommand{\GG}{\mathbb{G}}
\newcommand{\KK}{\mathbb{K}}

\newcommand{\sA}{\mathsf{A}}
\newcommand{\sB}{\mathsf{B}}
\newcommand{\sX}{\mathsf{X}}
\newcommand{\sY}{\mathsf{Y}}

\newcommand{\tDelta}{\Delta^{\!**\!}}

\newcommand{\bh}{\boldsymbol{h}}

\begin{document}

\title[Kawada-It\^o-Kelley Theorem]{Kawada-It\^o-Kelley Theorem for Quantum Semigroups}

\author{Pawe{\l} Kasprzak}
\address{Department of Mathematical Methods in Physics, Faculty of Physics, University of Warsaw, Poland}
\email{pawel.kasprzak@fuw.edu.pl}

\author{Fatemeh Khosravi}
\address{School of Mathematics, Institute for Research in Fundamental Sciences (IPM), P.O. box: 19395--5746, Tehran, Iran}
\email{f.khosravi@ipm.ir}

\author{Piotr M.~So{\l}tan}
\address{Department of Mathematical Methods in Physics, Faculty of Physics, University of Warsaw, Poland}
\email{piotr.soltan@fuw.edu.pl}

\begin{abstract}
Idempotent states on locally compact quantum semigroups with weak cancellation properties are shown to be Haar states on a certain sub-object described by an operator system with comultiplication. We also give a characterization of the situation when this sub-object is actually a compact quantum subgroup. In particular we reproduce classical results on idempotent probability measures on locally compact semigroups with cancellation.
\end{abstract}

\maketitle

\section{Introduction}\label{Intro}

Idempotent measures on locally compact groups are a classical topic on the intersection of probability theory and harmonic analysis. The book of Heyer \cite{Heyer} contains a detailed account of the progress on characterization of idempotent measures on more and more general classes of groups (\cite[Page 79]{Heyer}). The final result which in its most general form is due to Kelley (\cite{Kelley}) and was earlier proved by Rudin, Cohen (for abelian locally compact groups \cite{Rudin,Cohen}) and Kawada and It\^o (for compact groups \cite{Kawada-Ito}) states that a probability measure $\mu$ on a locally compact group $G$ which satisfies $\mu\star\mu=\mu$ must necessarily arise as the Haar measure of a compact subgroup of $G$. We will refer to this result as the Kawada-It\^o-Kelley theorem.

Quite early in the development of the theory of compact quantum groups it was noticed by Pal in \cite{Pal} that the analogous result is not true for compact (and even for finite) quantum groups. Pal provides an example of an idempotent state on the Kac-Paljutkin quantum group (\cite{Kac-Paljutkin}) which does not come from a quantum subgroup (it is easy to see that also the dual of any finite group with a non-normal subgroup will provide an example of such a state).

Soon it was realized that the set of idempotent states on a given locally compact quantum group $\GG$ behaves in a way analogous to the set of all compact quantum subgroups of $\GG$. This idea was pursued e.g.~in \cite{Baaj-Blanchard-Skandalis,Franz-Skalski2009} for finite quantum groups and in \cite{lattice} for all locally compact quantum groups.

The aim of this paper is to extend \cite[Theorem 4.4]{Franz-Skalski2009} to objects which we call quantum semigroups with weak cancellation. In particular this class contains all locally compact quantum groups.

Theory of idempotent measures on a locally compact semigroups was developed in a number of papers, see e.g.~\cite{pym1962}, \cite{Sun-Tserpes} and culminates in \cite{Mukherjea_Tserpes}. The latter contains the following structural characterization of idempotent measures: let $S$ be a locally compact semigroup and $\omega\in\C_0(S)^*$ an idempotent state with the support $S_\omega\subset{S}$. Then $S_\omega$ is a closed subsemigroup and there are locally compact spaces $X,Y$ equipped with Borel probability measures $\mu_X$ and $\mu_Y$, a compact group $G$, a map $\phi:Y\times{X}\to{G}$ such that:
\begin{itemize}
\item $S_\omega$ can be identified with $X\times{G}\times{Y}$ with product in the latter semigroup is defined by $(x_1,g_1,y_1)(x_2,g_2,y_2)=\bigl(x_1,g_1\phi(y_1,x_2)g_2,y_2\bigr)$;
\item under the above identification $\omega$ is of the form $\mu_X\times\mu_G\times\mu_Y$, where $\mu_G$ is the Haar measure on $G$.
\end{itemize}
Note that if $S$ has cancellation laws then $X$ and $Y$ must be singletons. In particular the support of $\omega$ is a compact subgroup of $S$ and $\omega$ is the Haar measure on it. It follows that the obvious version of Kawada-It\^o-Kelley theorem holds in this case.

In our work we study idempotent states on locally compact quantum semigroups satisfying the analogs of cancellation laws. Although we formulate a relatively weak condition implying compactness of the support of such an idempotent state (cf.~Proposition \ref{ecpt}), we were not able to prove this compactness in general.

The next definition uses the notion of a ``quantum space'' defined as an object of the category dual to the category of \cst-algebras with \emph{morphisms of \cst-algebras} (see Section \ref{cstSect}) as morphisms. In this language the phrase ``let $\mathbb{X}$ be a quantum space'' has exactly the same meaning as ``let $\C_0(\mathbb{X})$ be a \cst-algebra''. The tensor product of \cst-algebras and more generally of operator systems is the minimal one throughout the paper.

\begin{definition}
A \emph{quantum semigroup} is a quantum space $\SS$ such that the \cst-algebra $\C_0(\SS)$ is equipped with $\Delta_\SS\in\Mor(\C_0(\SS),\C_0(\SS)\tens\C_0(\SS))$ which is \emph{coassociative} i.e.
\[
(\Delta_\SS\tens\id)\comp\Delta_\SS=(\id\tens\Delta_\SS)\comp\Delta_\SS.
\]
\end{definition}

If $\SS$ is a quantum semigroup then the Banach space $\C_0(\SS)^*$ can be equipped with a Banach algebra structure provided by the \emph{convolution product}, i.e.~the map
\[
\C_0(\SS)^*\times\C_0(\SS)^*\ni(\mu,\nu)\longmapsto\mu\star\nu=(\mu\tens\nu)\comp\Delta_\SS.
\]
In particular an \emph{idempotent state} on $\SS$ is by definition a state $\omega\in\C_0(\SS)^*$ such that $\omega\star\omega=\omega$. In what follows we will also use the standard $\C_0(\SS)$-bimodule structure on $\C_0(\SS)^*$ given by
\[
(a\cdot\mu)(b)=\mu(ba)=(\mu\cdot{b})(a),\qqquad{a,b}\in\C_0(\SS),\mu\in\C_0(\SS)^*.
\]
There are also the operations of convolution of elements of $\C_0(\SS)^*$ and $\C_0(\SS)$ defined by
\begin{align*}
\C_0(\SS)^*\times\C_0(\SS)\ni(\mu,a)&\longmapsto\mu\star{a}=(\id\tens\mu)\Delta_\SS(a)\in\M(\C_0(\SS)),\\
\C_0(\SS)\times\C_0(\SS)^*\ni(b,\nu)&\longmapsto{b}\star\nu=(\nu\tens\id)\Delta_\SS(b)\in\M(\C_0(\SS))
\end{align*}
(see Section \ref{IdemSt}).

\begin{definition}
Let $\SS$ be a quantum semigroup. We say that $\SS$
\begin{enumerate}
\item has \emph{weak cancellation} if for any non zero $\mu\in\C_0(\SS)^*$ the sets
\[
\bigl\{b\star(\mu\cdot{a})\,\bigl|\bigr.\,a,b\in\C_0(\SS)\bigr\}
\quad\text{and}\quad
\bigl\{(b\cdot\mu)\star{a}\,\bigl|\bigr.\,a,b\in\C_0(\SS)\bigr\}
\]
are linearly strictly dense in $\M(\C_0(\SS))$,
\item has \emph{proper cancellation} if the sets
\[
\bigl\{(a\tens\I)\Delta_\SS(b)\,\bigl|\bigr.\,a,b\in\C_0(\SS)\bigr\}
\quad\text{and}\quad
\bigl\{\Delta_\SS(a)(\I\tens{b})\,\bigl|\bigr.\,a,b\in\C_0(\SS)\bigr\}
\]
are contained and linearly norm-dense in $\C_0(\SS)\tens\C_0(\SS)$.
\end{enumerate}
\end{definition}

The definition of weak cancellation laws was first given by G.J.~Murphy and L.~Tuset in \cite[Section 2]{Murphy-Tuset} in the context of \emph{compact} quantum semigroups. They also proved that for compact quantum semigroups these cancellation laws were equivalent to proper cancellation laws. Note also that for general locally compact quantum semigroups proper cancellation obviously implies weak cancellation.

Let us comment that quantum semigroups with proper cancellation were referred to ``bisimplifiable Hopf \cst-algebras'' in \cite{Baaj-Skandalis} and were called ``proper \cst-bialgebras with cancellation property'' in \cite{Masuda-Nakagami-Woronowicz}. This class includes quantum semigroups given by $(\C_0(\GG),\Delta_\GG)$ for all locally compact quantum groups as well as those defined by $(\C_0^{\text{\tiny{\rm{u}}}}(\GG),\Delta^{\text{\tiny{\rm{u}}}}_\GG)$ i.e.~the universal version of $\GG$ (see \cite{Kustermans}, or \cite{mmu} for a version without assumption of existence of Haar measures). Clearly there are many more examples. Let us note here that it is the weak cancellation property defined above that is equivalent to cancellation laws for classical semigroups:

\begin{proposition}\label{classC}
Let $S$ be a locally compact semigroup. Then $S$ has cancellation if and only if for any non-zero $\mu\in\C_0(S)^*$ the sets
\begin{equation}\label{classCancel}
\bigl\{g\star(\mu\cdot{f})\,\bigl|\bigr.\,f,g\in\C_0(S)\bigr\}
\quad\text{and}\quad
\bigl\{(g\cdot\mu)\star{f}\,\bigl|\bigr.\,f,g\in\C_0(S)\bigr\}
\end{equation}
are linearly strictly dense in $\M(\C_0(S))=\C_{\text{\tiny{\rm{b}}}}(S)$.
\end{proposition}

\begin{proof}
Cancellation properties for $S$ (from left and right) mean that the \emph{canonical maps}
\begin{equation}\label{classCanonical}
S\times{S}\ni(x,y)\longmapsto(x,xy)\in{S}\times{S}\quad\text{and}\quad
S\times{S}\ni(x,y)\longmapsto(xy,y)\in{S}\times{S}
\end{equation}
are injective, so the images of the associated morphisms $\Phi_{\text{\tiny{\rm{L}}}},\Phi_{\text{\tiny{\rm{R}}}}\in\Mor(\C_0(S)\tens\C_0(S),\C_0(S)\tens\C_0(S))$ are strictly dense (cf.~\cite[Theorem 1.1]{DKSS}) in $\M(\C_0(S)\tens\C_0(S))$. Since these images are precisely the closed linear spans of
\[
\bigl\{(f\tens\I)\Delta_S(g)\,\bigl|\bigr.\,f,g\in\C_0(S)\bigr\}
\quad\text{and}\quad
\bigl\{\Delta_S(f)(\I\tens{g})\,\bigl|\bigr.\,f,g\in\C_0(S)\bigr\}
\]
and for any $\mu\in\C_0(S)^*$ we have
\[
\bigl(g\star(\mu\cdot{f})\bigr)=(\mu\tens\id)\bigl((f\tens\I)\Delta_S(g)\bigr)
\quad\text{and}\quad
\bigl((g\cdot\mu)\star{f}\bigr)=(\id\tens\mu)\bigl(\Delta_S(f)(\I\tens{g})\bigr).
\]
Thus we see that if $S$ has cancellation and $\mu\neq{0}$ then the sets \eqref{classCancel} are linearly strictly dense in $\M(\C_0(S))$.

Conversely, assume that $S$ does not have cancellation (e.g.~from the right) so that there exists $x,x',y\in{S}$ such that $x\neq{x'}$ and $xy=x'y$. Taking $\mu$ equal to the evaluation functional at $y$ we have
\[
\bigl((g\cdot\mu)\star{f}\bigr)=f(\text{\large\textvisiblespace}\,y)g(y)
\]
so the strict closure of $\bigl\{(g\cdot\mu)\star{f}\,\bigl|\bigr.\,f,g\in\C_0(S)\bigr\}$ is contained in the subspace of bounded continuous functions which attain the same value at $x$ and $x'$.
\end{proof}

The property of having proper cancellation is certainly stronger then weak cancellation, since for a classical semigroup $S$ it implies that the canonical maps \eqref{classCanonical}
\[
S\times{S}\ni(x,y)\longmapsto(xy,y)\in{S}\times{S}\quad\text{and}\quad
S\times{S}\ni(x,y)\longmapsto(x,xy)\in{S}\times{S}
\]
are proper which need not be the case for general semigroups with cancellation. Indeed, a counterexample is provided by the locally compact semigroup $\left]0,+\infty\right[$ with addition which has cancellation, but whose canonical maps are not proper.

The main result of this paper is that an idempotent state on a quantum semigroup $\SS$ with weak cancellation which has compact support (cf.~Section \ref{cstSect}) is in a canonical way the Haar state on a rather primitive sub-object of $\SS$ which carries the structure of a \emph{proto-compact quantum hypergroup}:

\begin{definition}\label{Defproto}
A \emph{proto-compact quantum hypergroup} is described by a triple $(\sX,\Delta_\sX,\bh_\sX)$ consisting of an operator system $\sX$, a coassociative completely positive unital map $\Delta_\sX:\sX\to\sX\tens\sX$ and a faithful state $\bh_\sX$ on $\sX$ which satisfies
\[
(\bh_\sX\tens\id)\Delta_\sX(x)=\bh_\sX(x)\I=(\id\tens\bh_\sX)\Delta_\sX(x),\qqquad{x}\in\sX.
\]
We refer to $\bh_\sX$ as the \emph{Haar state}.\footnote{Uniqueness of of the Haar state is clear if $\bh_\sX'$ satisfies analogous conditions to those for $\bh_\sX$ then $\bh_\sX=\bh_\sX\star\bh_\sX'=\bh_\sX'$.}
\end{definition}

Definition \ref{Defproto} is motivated by the definition of a compact quantum hypergroup proposed by Chapovsky and Vainerman in \cite{Chapovsky-Vainerman}. The sophisticated structure of compact quantum hypergroups (see \cite[Section 4]{Chapovsky-Vainerman}) emphasized in that paper does not seem to be relevant to questions pertaining to idempotent states (cf.~\cite[Remarks after Proposition 1.10]{Franz-Skalski2009}).

In somewhat imprecise terms we can state that any idempotent state on a quantum semigroup $\SS$ with weak cancellation which has compact support is the Haar state on a ``sub-proto-compact quantum hypergroup''. One needs to stress the fact that representation of a given compactly supported idempotent state $\omega$ on a quantum semigroup with weak cancellation $\SS$ is not unique, since it is particularly easy to represent $\omega$ e.g.~as the Haar state on a trivial proto-compact quantum hypergroup (see Section \ref{kik}). This point was already addressed in \cite[Section 1]{Franz-Skalski2009} and following the ideas of \cite{Franz-Skalski2009} we also construct an appropriately universal representation in Section \ref{kik}.

In Section \ref{HaarType} we characterize the situation when the given idempotent state is of \emph{Haar type}, i.e.~it is the Haar measure on some compact quantum subgroup of $\SS$.

\section{Preliminaries}\label{cstSect}

Throughout the paper we will be using the language of \cst-algebras and von Neumann algebras. We refer e.g.~to \cite{Pedersen,Stratila-Zsido}. Additionally we will need some theory of operator systems, completely positive and completely bounded maps which can be found e.g.~in \cite{Paulsen}. Let us also state that we will be using exclusively the minimal tensor product of \cst-algebras and operator systems.

Much attention will be paid to multiplier algebras, strict topology and morphisms of \cst-algebras as defined in \cite[Section 0]{unbo} (cf.~also \cite{Lance}). More precisely if $\sA$ and $\sB$ are \cst-algebras then the set of \emph{morphisms} from $\sA$ to $\sB$ denoted $\Mor(\sA,\sB)$ is the set of $*$-homomorphisms $\Phi:\sA\to\M(\sB)$ which are non-degenerate, i.e.~the strict closure of the range of $\Phi$ contains $\I\in\M(\sB)$. All such maps extend canonically to $\M(\sA)$ and it is well known that the extensions are strictly continuous on bounded sets (\cite[Proposition 2.5]{Lance}). However, as explained in \cite[Section 2]{Salmi_strict} a linear map from a \cst-algebra to a locally convex space is strictly continuous if and only if it is strictly continuous on bounded sets (by \cite[Corollary 2.7]{Taylor}). Similarly any bounded linear functional $\ph$ on a \cst-algebra $\sA$ extends to a strictly continuous functional on $\M(\sA)$ (also denoted by $\ph$) and slice maps $(\id\tens\ph)$ and $(\ph\tens\id)$ also extend to strictly continuous maps $\M(\sA\tens\sA)\to\M(\sA)$.

The extensions of various kinds of maps from a \cst-algebra $\sA$ to its multiplier algebra have also a very useful picture drawn in terms of the second dual of $\sA$. More precisely let $\sA$ be a \cst-algebra. It is well known that $\sA^{**}$ is a von Neumann algebra isomorphic to the strong closure of the image of $\sA$ in a universal representation (i.e.~any faithful representation in which every continuous functional on $\sA$ is normal, cf.~\cite[Section 3.7]{Pedersen}). The canonical embedding $\sA\hookrightarrow{\sA^{**}}$ is then a non-degenerate $*$-homomorphism and one can identify $\M(\sA)$ with a subset of $\sA^{**}$ as follows
\[
\M(\sA)=\bigl\{x\in{\sA^{**}}\,\bigl|\bigr.\,x\sA,\sA{x}\subset\sA\bigr\}.
\]
Now elements of $\sA^*$ are normal functionals on $\sA^{**}$ and as such can be readily applied to $\M(\sA)$. The universal property of $\sA^{**}$ is that any $*$-homomorphism from $\sA$ to a von Neumann algebra has the \emph{normal extension} to $\sA^{**}$ (see \cite[Proposition 3.7.7]{Pedersen}). Slice maps also have a good description in terms of $\sA^{**}$. Indeed, the tensor product of the universal representations of $\sA$ is a faithful representation of $\sA\tens\sA$ on a Hilbert space whose strong closure is $\sA^{**}\operatorname{\bar{\tens}}\sA^{**}$. It follows that $\M(\sA\tens\sA)$ can be canonically identified with a subset of $\sA^{**}\operatorname{\bar{\tens}}\sA^{**}$ (the image of the extension of the representation of $\sA\tens\sA$ under consideration to the multiplier algebra).

The above remarks allow us to use a number of formulas involving elements of $\sA,\M(\sA),\sA^{**}$ as well as $\sA\tens\sA,\M(\sA\tens\sA)$ and continuous linear functionals on $\sA$. For example we have
\[
x\bigl((\ph\tens\id)Y\bigr)z=(\ph\tens\id)\bigl((\I\tens{x})Y(\I\tens{z})\bigr),\qqquad{Y}\in\M(\sA\tens\sA),\:{x,z}\in\sA^{**},\:\ph\in\sA^*.
\]
In particular if $\psi\in\sA^*$ then
\[
\psi\bigl(x\bigl((\ph\tens\id)Y\bigr)z)=\psi\Bigl((\ph\tens\id)\bigl((\I\tens{x})Y(\I\tens{z})\bigr)\bigr)\Bigr).
\]

Let us now recall the notion of a compact projection in the second dual of a \cst-algebra (\cite{Akemann} see also \cite{Akemann-Pedersen-Tomiyama,Blecher-Neal}). So let $\sA$ be a \cst-algebra. A projection $q\in\sA^{**}$ is called \emph{open} if $q$ is a strong limit of elements of $\sA\subset\sA^{**}$. A complement $\I-q$ of an open projection is by definition a \emph{closed} projection. Finally, a closed projection $p\in\sA^{**}$ is \emph{compact} if there exists an element $x\in\sA_+$ such that $xp=p$. In this case one can take $x$ to be of norm $1$.

\section{Idempotent states}\label{IdemSt}

Let $\SS$ be a quantum semigroup. As already mentioned in the Introduction the space $\C_0(\SS)^*$ is naturally a Banach algebra under the \emph{convolution product}:
\[
(\mu,\nu)\longmapsto\mu\star\nu=(\mu\tens\nu)\comp\Delta_\SS,\qqquad\mu,\nu\in\C_0(\SS)^*
\]
and a state $\omega$ on $\C_0(\SS)$ is called \emph{idempotent}, if it satisfies $\omega\star\omega=\omega$.

Note that if $\SS$ does not have proper cancellation (at least from the right) then $\mu\star{a}$ does not necessarily belong to $\C_0(\SS)$, but only to $\M(\C_0(\SS))$. Still the mapping
\[
\C_0(\SS)\ni{a}\longmapsto\mu\star{a}\in\M(\C_0(\SS))
\]
is strictly continuous and extends to a strictly continuous map $\M(\C_0(\SS))\to\M(\C_0(\SS))$. If $\mu$ is positive then it is completely positive and if $\mu$ is a state then it is unital. Finally if $\omega$ is idempotent then the map $a\mapsto\omega\star{a}$ is idempotent. This is part of the proof of the following proposition which extends to quantum semigroups with weak cancellation \cite[Lemma 2.5]{Salmi-Skalski} earlier proved for compact quantum groups in \cite{Franz-Skalski_new}. Our conventions are slightly different and the proof of \cite[Lemma 3.1]{Franz-Skalski_new} is only available in the extended electronic version of that paper \cite{Franz-Skalski_new_e}. We have therefore decided to include most of the steps in our slightly more general version.

\begin{proposition}\label{MultDom}
Let $\SS$ be a quantum semigroup with weak cancellation and let $\omega\in\C_0(\SS)^*$ be an idempotent state. Then the ranges of the maps $\M(\C_0(\SS))\ni{a}\mapsto\omega\star{a}\in\M(\C_0(\SS))$ and $\M(\C_0(\SS))\ni{a}\mapsto{a}\star\omega\in\M(\C_0(\SS))$ are contained in the multiplicative domain of $\omega$.
\end{proposition}

\begin{proof}
We will only prove that for any $a\in\M(\C_0(\SS))$ the element $\omega\star{a}$ is in the multiplicative domain of $\omega$. The fact that $a\star\omega$ is then also in the multiplicative domain of $\omega$ follows from the fact that $\omega$ is also an idempotent state on $\SS^{\text{\tiny{\rm{op}}}}$, i.e.~the quantum semigroup $\SS$ with opposite comultiplication.

Exactly as in \cite[Proof of Lemma 3.1]{Franz-Skalski_new_e} take first $b\in\C_0(\SS)$ and put $y=b\star\omega$. Then
\[
(\omega\tens\omega)\bigl(\Delta_\SS(y^*)(\I\tens{y})\bigr)=\omega\Bigl(\bigl(b^*\star(\omega\star\omega)\bigr)(b\star\omega)\Bigr)
=\omega\bigl((b^*\star\omega)(b\star\omega)\bigr)=\omega(y^*y)
\]
and by conjugation also $(\omega\tens\omega)\bigl((\I\tens{y^*})\Delta_\SS(y)\bigr)=\omega(y^*y)$. Therefore
\begin{align*}
(\omega\tens\omega)&\Bigl(\bigl(\Delta_\SS(y)-(\I\tens{y})\bigr)^*\bigl(\Delta_\SS(y)-(\I\tens{y})\bigr)\Bigr)\\
&=(\omega\tens\omega)\bigl(\Delta_\SS(y^*y)\bigr)-(\omega\tens\omega)\bigl((\I\tens{y^*})\Delta_\SS(y)\bigr)\\&\qquad-(\omega\tens\omega)\bigl(\Delta_\SS(y^*)(\I\tens{y})\bigr)+\omega(y^*y)\\
&=\omega(y^*y)-\omega(y^*y)-\omega(y^*y)+\omega(y^*y)=0.
\end{align*}
Hence, by Choi's Cauchy-Schwarz inequality (\cite[Proposition 3.3]{Paulsen}), for any $a,c\in\C_0(\SS)$
\begin{align*}
\Bigl|(\omega\tens\omega)&\Bigl((c\tens{a})\bigl(\Delta_\SS(y)-(\I\tens{y})\bigr)\Bigr)\Bigr|^2\\&\leq
(\omega\tens\omega)(cc^*\tens{aa^*})(\omega\tens\omega)\Bigl(\bigl(\Delta_\SS(y)-(\I\tens{y})\bigr)^*\bigl(\Delta_\SS(y)-(\I\tens{y})\bigr)\Bigr)=0
\end{align*}
which reads
\[
(\omega\tens\omega)\bigl((c\tens{a})\Delta_\SS(y)\bigr)=(\omega\tens\omega)(c\tens{ay}).
\]
Returning to $b$ we can rewrite this as
\[
\bigl(\omega\star(\omega\cdot{c})\bigr)\bigl((\omega\cdot{a})\star b)\bigr)=\omega(c)\omega\bigl((\omega\cdot{a})\star b\bigr),\qqquad{a,b,c}\in\C_0(\SS).
\]
Now the weak cancellation (from the right) asserts that the set $\bigl\{(\omega\cdot{a})\star b\,\bigl|\bigr.\,a,b\in\C_0(\SS)\bigr\}$ is linearly strictly dense in $\M(\C_0(\SS))$, so we find that $\omega\star(\omega\cdot{c})=\omega(c)\omega$ for all $c\in\C_0(\SS)$ and hence, by strict continuity
\begin{equation}\label{pierwszy0}
\omega\star(\omega\cdot{c})=\omega(c)\omega,\qqquad{c}\in\M(\C_0(\SS)).
\end{equation}
Similarly, using weak cancellation from the left we show that
\begin{equation}\label{pierwszy}
(\omega\cdot{c})\star\omega=\omega(c)\omega,\qqquad{c}\in\M(\C_0(\SS)).
\end{equation}

Using \eqref{pierwszy} we repeat the following elements of the proof of \cite[Theorem 2.5]{Salmi-Skalski}: take any $a,b,c\in\C_0(\SS)$. Then
\begin{align*}
(\id\tens\omega\tens\omega)\Bigl((b\tens{c}\tens\I)\bigl((\Delta_\SS\tens\id)\Delta_\SS(a)\bigr)\Bigr)&=\bigl(\id\tens(\omega\cdot{c})\tens\omega\bigr)\Bigl((b\tens\I\tens\I)\bigl((\Delta_\SS\tens\id)\Delta_\SS(a)\bigr)\Bigr)\\
&=\bigl(\id\tens(\omega\cdot{c})\tens\omega\bigr)\Bigl((b\tens\I\tens\I)\bigl((\id\tens\Delta_\SS)\Delta_\SS(a)\bigr)\Bigr)\\
&=\Bigl(\id\tens\bigl((\omega\cdot{c})\star\omega\bigr)\Bigr)\bigl((b\tens\I)\Delta_\SS(a)\bigr)\\
&=\Bigl(\id\tens\bigl(\omega(c)\omega\bigr)\Bigr)\bigl((b\tens\I)\Delta_\SS(a)\bigr)\\
&=(\id\tens\omega\tens\omega)\bigl((b\tens{c}\tens\I)\Delta_\SS(a)_{13}\bigr)
\end{align*}
Now by strict continuity we can replace $(b\tens{c})$ by any element of $\M(\C_0(\SS)\tens\C_0(\SS))$. Taking this element to be $\Delta_\SS(x)$ for some $x\in\C_0(\SS)$ we obtain
\begin{align*}
\omega\star\bigl(x(\omega\star{a})\bigr)&=(\id\tens\omega\tens\omega)\Bigl(\bigl(\Delta_\SS(x)\tens\I\bigr)\bigl((\Delta_\SS\tens\id)\Delta_\SS(a)\bigr)\Bigr)\\
&=(\id\tens\omega\tens\omega)\bigl(\Delta_\SS(x)_{12}\Delta_\SS(a)_{13}\bigr)=
(\omega\star{x})(\omega\star{a}).
\end{align*}
Using this with $x=a^*$ and with help from \eqref{pierwszy} we obtain
\begin{align*}
\omega\bigl((\omega\star{a})^*(\omega\star{a})\bigr)&=\omega\bigl((\omega\star{a^*})(\omega\star{a})\bigr)=\omega\Bigl(\omega\star\bigl(a^*(\omega\star{a})\bigr)\Bigr)=(\omega\star\omega)\bigl(a^*(\omega\star{a})\bigr)\\
&=\omega\bigl(a^*(\omega\star{a})\bigr)=(\omega\cdot{a^*})(\omega\star{a})=\bigl((\omega\cdot{a^*})\star\omega\bigr)(a)=\omega(a^*)\omega(a)\\
&=\omega(\omega\star{a^*})\omega(\omega\star{a})=\omega\bigl((\omega\star{a})^*\bigr)\omega(\omega\star{a})=\overline{\omega(\omega\star{a})}\omega(\omega\star{a})
\end{align*}
and upon substituting $a^*$ for $a$ also $\omega\bigl((\omega\star{a})(\omega\star{a})^*\bigr)=\omega(\omega\star{a})\overline{\omega(\omega\star{a})}$. By \cite[Theorem 3.18]{Paulsen} this guarantees that $\omega\star{a}$ is in the multiplicative domain of $\omega$.
\end{proof}

\section{Kawada-It\^o-Kelley theorem for quantum semigroups with weak cancellation}\label{kik}

Let $\SS$ be a quantum semigroup with weak cancellation and let $\omega\in\C_0(\SS)^*$ be an idempotent state. Clearly $\omega$ defines a normal state on the von Neumann envelope $\C_0(\SS)^{**}$ of $\C_0(\SS)$ which we will denote by the same symbol. The left kernel $N_\omega$ of this state, i.e.~the set
\[
\bigl\{x\in\C_0(\SS)^{**}\,\bigl|\bigr.\,\omega(x^*x)=0\bigr\}
\]
is a weak${}^*$-closed left ideal and hence there exists a unique projection $p_\omega\in\C_0(\SS)^{**}$ such that $N_\omega=\C_0(\SS)^{**}p_\omega$. This projection is the weak${}^*$-limit of any left approximate unit for the left ideal
\[
C_\omega=\bigl\{x\in\C_0(\SS)\,\bigl|\bigr.\,\omega(x^*x)=0\bigr\}
\]
of $\C_0(\SS)$ (cf.~\cite[Section 3.20]{Stratila-Zsido}) and hence it is open. We will denote the orthogonal complement $\I-p_\omega$ of $p_\omega$ by the symbol $p_\omega^\perp$. Clearly $p_\omega^\perp$ is a closed projection. Note also that we have
\[
C_\omega=N_\omega\cap\C_0(\SS)=\C_0(\SS)^{**}p_\omega\cap\C_0(\SS)
=\bigl\{a\in\C_0(\SS)\,\bigl|\bigr.\,ap_\omega^\perp=0\bigr\}.
\]
We will refer to $p_\omega^\perp$ as the \emph{support} of $\omega$. Note that in particular we have $p_\omega^\perp\cdot\omega=\omega\cdot{p_\omega^\perp}=\omega$.

Our standing assumption is that the support of $\omega$ is compact. We do not know whether this assumption is at all restrictive because we do not know a single example of an idempotent state whose support is not compact. We can prove that the support of an idempotent state is compact in many cases including the case when the quantum semigroup under consideration has proper cancellation (for example it is a locally compact quantum group). This is a corollary of the following proposition:

\begin{proposition}\label{ecpt}
Let $\SS$ be a quantum semigroup with weak cancellation and let $\omega\in\C_0(\SS)^*$ be an idempotent state. Assume that there is a positive element $e\in\C_0(\SS)$ such that $\omega(e)=1$ and $\omega\star{e}\in\C_0(\SS)$. Then $p_\omega^\perp$ is compact.
\end{proposition}

\begin{proof}
By Proposition \ref{MultDom} for each $e\in\C_0(\SS)$ the element $\omega\star{e}$ belongs to the multiplicative domain of $\omega$ and obviously $\omega(\omega\star{e})=\omega(e)$. For every $z\in\C_0(\SS)^{**}$ we have
\begin{align*}
\omega\Bigl(\bigl(z-z(\omega\star{e})\bigr)^*\bigl(z-z(\omega\star{e})\bigr)\Bigr)
&=\omega\bigl(z^*z-z^*z(\omega\star{e})-(\omega\star{e})z^*z+(\omega\star{e})z^*z(\omega\star{e})\bigr)\\
&=\omega(z^*z)\bigl(1-\omega(e)-\omega(e)+\omega(e)^2\bigr)=0.
\end{align*}
It follows that $z-z(\omega\star{e})\in{N_\omega}$, which proves that
\begin{equation}\label{allzp}
zp_\omega^\perp=z(\omega\star{e})p_\omega^\perp,\qqquad{z}\in\C_0(\SS)^{**}.
\end{equation}
Inserting $z=p_\omega^\perp$ in \eqref{allzp} gives $p_\omega^\perp(\omega\star{e})p_\omega^\perp=(\omega\star{e})p_\omega^\perp$ and $z=\I$ gives $p_\omega^\perp=(\omega\star{e})p_\omega^\perp$. Since $\omega\star{e}\in\C_0(\SS)$ we see that $p_\omega^\perp$ is a compact projection.
\end{proof}

\begin{remark}\label{Rem_proper}
Let $\SS$ be a quantum semigroup with weak cancellation and let $\omega\in\C_0(\SS)^*$ be an idempotent state.
\begin{enumerate}
\item As we noted before Proposition \ref{ecpt} if $\SS$ has proper cancellation then for any $e\in\C_0(\SS)$ we have $\omega\star{e}\in\C_0(\SS)$. In particular if $\SS=\GG$ or $\SS=\GG^{\text{\tiny{\rm{u}}}}$ for a locally compact quantum group $\GG$ then $p_\omega^\perp$ is a compact projection.
\item\label{Rem_proper2} The fact that the support of an idempotent state on a quantum semigroup with proper cancellation is always compact can be considered a simple Kawada-It\^o-Kelley type result.
\end{enumerate}
\end{remark}

\begin{proposition}\label{prop_perp_belongs}
Let $\SS$ be a quantum semigroup with weak cancellation and let $\omega\in\C_0(\SS)^*$ be an idempotent state with compact support $p_\omega^\perp$. Then the projection $p_\omega^\perp$ belongs to $p_\omega^\perp\C_0(\SS)p_\omega^\perp$ and consequently $p_\omega^\perp\C_0(\SS)p_\omega^\perp$ is an operator system in $p_\omega^\perp\C_0(\SS)^{**}p_\omega^\perp$. Conversely, if $p_\omega^\perp\in p_\omega^\perp\C_0(\SS)p_\omega^\perp$ then $p_\omega^\perp$ is a compact projection.
\end{proposition}

\begin{proof}
Since $p_\omega^\perp=xp_\omega^\perp$ for some $x\in\C_0(\SS)$, we have $p_\omega^\perp=p_\omega^\perp{x}p_\omega^\perp\in{p_\omega^\perp\C_0(\SS)p_\omega^\perp}$. Clearly $p_\omega^\perp\C_0(\SS)p_\omega^\perp$ is a norm-closed self-adjoint subspace of the unital \cst-algebra $p_\omega^\perp\C_0(\SS)^{**}p_\omega^\perp$ containing the unit of the latter. Conversely, if $p_\omega^\perp\in p_\omega^\perp\C_0(\SS)p_\omega^\perp$ then $p_\omega^\perp$ is compact due to \cite[Theorem 7.2]{Blecher-Neal}.
\end{proof}

In what follows we will denote the operator system $p_\omega^\perp\C_0(\SS)p_\omega^\perp$ by $\sX_\omega$. Our next aim is to define a comultiplication on $\sX_\omega$. For this let us recall that $\M(\C_0(\SS)\tens\C_0(\SS))$ can be identified with a subset of $\C_0(\SS)^{**}\operatorname{\bar{\tens}}\C_0(\SS)^{**}$. Therefore we can consider the \emph{normal extension} of the $*$-homomorphism $\Delta_\SS:\C_0(\SS)\to\M(\C_0(\SS)\tens\C_0(\SS))$ to a normal unital $*$-homomorphism $\C_0(\SS)^{**}\to\C_0(\SS)^{**}\operatorname{\bar{\tens}}\C_0(\SS)^{**}$. We will denote this extension by $\tDelta$. Clearly $\bigl.\tDelta\bigr|_{\C_0(\SS)}=\Delta_\SS$ and moreover if $\omega\in\C_0(\SS)^*$ is an idempotent state then by strong density of $\C_0(\SS)$ in $\C_0(\SS)^{**}$ we also have
\[
(\omega\tens\omega)\comp\tDelta=\omega.
\]

\begin{theorem}\label{DeltaOnX}
Let $\SS$ be a quantum semigroup with weak cancellation and let $\omega\in\C_0(\SS)^*$ be an idempotent state with compact support $p_\omega^\perp$. Then the restriction $\Delta_\omega$ of the mapping
\begin{equation}\label{DeltaBig}
\C_0(\SS)^{**}\ni{y}\longmapsto(p_\omega^\perp{\tens}p_\omega^\perp)\tDelta(y)(p_\omega^\perp{\tens}p_\omega^\perp)\in\C_0(\SS)^{**}\operatorname{\bar{\tens}}\C_0(\SS)^{**}
\end{equation}
to $\sX_\omega$ has its range in $\sX_\omega\tens\sX_\omega$ and $\Delta_\omega$ is unital, completely positive and coassociative.
\end{theorem}

\begin{proof}
Clearly \eqref{DeltaBig} is a completely positive map. In order to prove the properties of $\Delta_\omega$ we will first show that
\begin{equation}\label{Dp0}
\tDelta(p_\omega^\perp)(p_\omega^\perp{\tens}p_\omega^\perp)=p_\omega^\perp{\tens}p_\omega^\perp.
\end{equation}
To that end consider the positive element $(\id\tens\omega)\tDelta(p_\omega)$. We have
\[
\omega\bigl((\id\tens\omega)\tDelta(p_\omega)\bigr)=(\omega\tens\omega)\tDelta(p_\omega)=\omega(p_\omega)=0,
\]
so $p_\omega^\perp\bigr((\id\tens\omega)\tDelta(p_\omega)\bigr)p_\omega^\perp=0$. Thus for any positive $\mu\in\C_0(\SS)^*$
\[
0=\mu\Bigl(p_\omega^\perp\bigr((\id\tens\omega)\tDelta(p_\omega)\bigr)p_\omega^\perp\Bigr)=\omega\Bigl((\mu\tens\id)\bigl((p_\omega^\perp\tens\I)\tDelta(p_\omega)(p_\omega^\perp\tens\I)\bigr)\Bigr),
\]
so that
\[
(\mu\tens\id)\Bigl((p_\omega^\perp\tens{p_\omega^\perp})\tDelta(p_\omega)(p_\omega^\perp\tens{p_\omega^\perp})\Bigr)=p_\omega^\perp\Bigl((\mu\tens\id)\bigl((p_\omega^\perp\tens\I)\tDelta(p_\omega)(p_\omega^\perp\tens\I)\bigr)\Bigr)p_\omega^\perp=0.
\]
Since this holds for all positive $\mu$, we obtain $(p_\omega^\perp\tens{p_\omega^\perp})\tDelta(p_\omega)(p_\omega^\perp\tens{p_\omega^\perp})=0$ or in other words
\[
(p_\omega^\perp\tens{p_\omega^\perp})\tDelta(p_\omega)\tDelta(p_\omega)(p_\omega^\perp\tens{p_\omega^\perp})=0.
\]
Thus $\tDelta(p_\omega)(p_\omega^\perp\tens{p_\omega^\perp})=0$ and we arrive at \eqref{Dp0}.

Take now any element $y=p_\omega^\perp{a}p_\omega^\perp$ of $\sX_\omega=p_\omega^\perp\C_0(\SS)p_\omega^\perp$ and let $x\in\C_0(\SS)$ be such that $p_\omega^\perp=xp_\omega^\perp$. Then using \eqref{Dp0} we obtain
\begin{align*}
(p_\omega^\perp{\tens}p_\omega^\perp)\tDelta(y)(p_\omega^\perp{\tens}p_\omega^\perp)&=(p_\omega^\perp{\tens}p_\omega^\perp)\tDelta(p_\omega^\perp{a}p_\omega^\perp)(p_\omega^\perp{\tens}p_\omega^\perp)\\
&=(p_\omega^\perp{\tens}p_\omega^\perp)\tDelta(p_\omega^\perp)\tDelta(a)\tDelta(p_\omega^\perp)(p_\omega^\perp{\tens}p_\omega^\perp)\\
&=(p_\omega^\perp{\tens}p_\omega^\perp)\tDelta(p_\omega^\perp)\Delta_\SS(a)\tDelta(p_\omega^\perp)(p_\omega^\perp{\tens}p_\omega^\perp)\\
&=(p_\omega^\perp{\tens}p_\omega^\perp)\Delta_\SS(a)(p_\omega^\perp{\tens}p_\omega^\perp)\\
&=(p_\omega^\perp{\tens}p_\omega^\perp)\Delta_\SS(a)(xp_\omega^\perp{\tens}xp_\omega^\perp)\\
&=(p_\omega^\perp{\tens}p_\omega^\perp)\bigl(\Delta_\SS(a)(x\tens{x})\bigr)(p_\omega^\perp{\tens}p_\omega^\perp)
\end{align*}
and the last element belongs to $\sX_\omega\tens\sX_\omega$. It follows that $\Delta_\omega(\sX_\omega)\subset\sX_\omega\tens\sX_\omega$ and that it is completely positive and unital.

The coassociativity of $\Delta_\omega$ follows now again from \eqref{Dp0}: for $a\in\C_0(\SS)$ we have
\begin{align*}
(\Delta_\omega&\tens\id)\Delta_\omega(p_\omega^\perp{a}p_\omega^\perp)\\
&=(p_\omega^\perp{\tens}p_\omega^\perp\tens\I)\Bigl(
(\tDelta\tens\id)\bigl((p_\omega^\perp{\tens}p_\omega^\perp)\tDelta(p_\omega^\perp{a}p_\omega^\perp)(p_\omega^\perp{\tens}p_\omega^\perp)\bigr)
\Bigr)(p_\omega^\perp{\tens}p_\omega^\perp\tens\I)\\
&=(p_\omega^\perp{\tens}p_\omega^\perp\tens\I)\Bigl(
(\tDelta\tens\id)\bigl((p_\omega^\perp{\tens}p_\omega^\perp)\tDelta(p_\omega^\perp)\Delta_\SS(a)\tDelta(p_\omega^\perp)(p_\omega^\perp{\tens}p_\omega^\perp)\bigr)
\Bigr)(p_\omega^\perp{\tens}p_\omega^\perp\tens\I)\\
&=(p_\omega^\perp{\tens}p_\omega^\perp{\tens}p_\omega^\perp)\bigl((\tDelta\tens\id)\Delta_\SS(a)\bigr)(p_\omega^\perp{\tens}p_\omega^\perp{\tens}p_\omega^\perp)\\
&=(p_\omega^\perp{\tens}p_\omega^\perp{\tens}p_\omega^\perp)\bigl((\Delta_\SS\tens\id)\Delta_\SS(a)\bigr)(p_\omega^\perp{\tens}p_\omega^\perp{\tens}p_\omega^\perp)\\
&=(p_\omega^\perp{\tens}p_\omega^\perp{\tens}p_\omega^\perp)\bigl((\id\tens\Delta_\SS)\Delta_\SS(a)\bigr)(p_\omega^\perp{\tens}p_\omega^\perp{\tens}p_\omega^\perp)
=(\id\tens\Delta_\omega)\Delta_\omega(p_\omega^\perp{a}p_\omega^\perp)
\end{align*}
where the last equality is obtained by reversing all previous steps.
\end{proof}

\begin{remark}\label{innywzor}
Let $\SS$ be a quantum semigroup with weak cancellation and let $\omega\in\C_0(\SS)^*$ be an idempotent state with compact support $p_\omega^\perp$. In Theorem \ref{DeltaOnX} we have shown that there is a comultiplication $\Delta_\omega:\sX_\omega\to\sX_\omega\tens\sX_\omega$. The proof also shows that there is a natural formula
\[
\Delta_\omega(p_\omega^\perp{a}p_\omega^\perp)=(p_\omega^\perp{\tens}p_\omega^\perp)\Delta_\SS(a)(p_\omega^\perp{\tens}p_\omega^\perp),\qqquad{a}\in\C_0(\SS).
\]
One can show that this formula defines a map on $\sX_\omega$ even without the assumption that $p_\omega^\perp$ be compact.
\end{remark}

\begin{theorem}\label{Xomega}
Let $\SS$ be a quantum semigroup with weak cancellation and let $\omega\in\C_0(\SS)^*$ be an idempotent state with compact support. Let $\bh_\omega$ be the restriction of $\omega$ to $\sX_\omega$. Then $(\sX_\omega,\Delta_\omega,\bh_\omega)$ is a proto-compact quantum hypergroup.
\end{theorem}

\begin{proof}
It is clear that $\bh_\omega$ is faithful, so we only need to prove that it is invariant, i.e.
\[
(\bh_\omega\tens\id)\Delta_\omega(x)=(\id\tens\bh_\omega)\Delta_\omega(x)=\bh(x)p_\omega^\perp,\qqquad{x}\in\sX_\omega.
\]

To prove right invariance we note that first that for any $a\in\C_0(\SS)$ we have
\begin{align*}
\omega\Bigl(\bigl(\omega(a)&p_\omega^\perp-(a\star\omega)\bigr)^*\bigl(\omega(a)p_\omega^\perp-(a\star\omega)\bigr)\Bigr)\\
&=\omega\Bigl(\omega(a^*)\omega(a)p_\omega^\perp-\omega(a^*)(a\star\omega)-(a^*\star\omega)\omega(a)
+(a^*\star\omega)(a\star\omega)\Bigr)\\
&=\omega(a^*)\omega(a)-\omega(a^*)\omega(a\star\omega)-\omega(a^*\star\omega)\omega(a)+\omega\bigl((a^*\star\omega)(a\star\omega)\bigr)\\
&=\omega(a^*)\omega(a)-\omega(a^*)\omega(a)-\omega(a^*)\omega(a)+\omega(a^*)\omega(a)=0
\end{align*}
because $a\star\omega$ is in the multiplicative domain of $\omega$ (Proposition \ref{MultDom}). This means that
\begin{equation}\label{ompom0}
p_\omega^\perp\bigl(\omega(a)p_\omega^\perp-(a\star\omega)\bigr)p_\omega^\perp=0
\end{equation}
and consequently
\begin{equation}\label{ompom}
p_\omega^\perp(a\star\omega)p_\omega^\perp=\omega(a)p_\omega^\perp=\omega(p_\omega^\perp{a}p_\omega^\perp)p_\omega^\perp=\bh_\omega(p_\omega^\perp{a}p_\omega^\perp)p_\omega^\perp.
\end{equation}

Using \eqref{ompom} and Remark \ref{innywzor} we compute:
\begin{align*}
(\bh_\omega\tens\id)\bigl(\Delta_\omega(p_\omega^\perp{a}p_\omega^\perp)\bigr)&=(\bh_\omega\tens\id)\bigl((p_\omega^\perp{\tens}p_\omega^\perp)\Delta_\SS(a)(p_\omega^\perp{\tens}p_\omega^\perp)\bigr)\\
&=(\omega\tens\id)\bigl((p_\omega^\perp{\tens}p_\omega^\perp)\Delta_\SS(a)(p_\omega^\perp{\tens}p_\omega^\perp)\bigr)\\
&=p_\omega^\perp\bigl((\omega\tens\id)\Delta_\SS(a)\bigr)p_\omega^\perp=p_\omega^\perp(a\star\omega)p_\omega^\perp=\bh_\omega(p_\omega^\perp{a}p_\omega^\perp)p_\omega^\perp.
\end{align*}
This proves right invariance of $\bh_\omega$. Left invariance is proved similarly.
\end{proof}

Interestingly the following "converse" of Theorem \ref{DeltaOnX} holds:

\begin{proposition}
Let $\SS$ be a quantum semigroup with weak cancellation and let $\omega\in\C_0(\SS)^*$ be an idempotent state. Denote by $p_\omega^\perp \in\C_0(\SS)^{**}$ the support of $\omega$ and $\sX_\omega=p_\omega^\perp\C_0(\SS)p_\omega^\perp$. Then $p_\omega^\perp$ is a compact projection if and only if there exists $e\in\C_0(\SS)$ such that $\omega(e)=1$ and
\begin{equation}\label{compact1}
(p_\omega^\perp\tens{p_\omega^\perp})\Delta_\SS(e)(p_\omega^\perp\tens{p_\omega^\perp})\in\sX_\omega\tens\sX_\omega.
\end{equation}
\end{proposition}
\begin{proof}
If $p_\omega^\perp$ is compact then $(p_\omega^\perp\tens{p_\omega^\perp})\Delta_\SS(e)(p_\omega^\perp\tens{p_\omega^\perp})\in\sX_\omega\tens\sX_\omega$ for every $e\in\C_0(\SS)$ due to Theorem \ref{DeltaOnX}.

Suppose conversely that there exists $e\in\C_0(\SS)$ such that $\omega(e)=1$ and \eqref{compact1} is satisfied. Applying $(\id\tens\omega)$ to \eqref{compact1} we get $p_\omega^\perp(\omega\star{e})p_\omega^\perp\in\sX_\omega$. Using Equation \eqref{ompom0} we get $p_\omega^\perp(\omega\star{e})p_\omega^\perp=\omega(e)p_\omega^\perp=p_\omega^\perp\in\sX_\omega$ and by Proposition \ref{prop_perp_belongs} we conclude that $p_\omega^\perp$ is compact.
\end{proof}

The facts established so far in this section are that if $\SS$ is a quantum semigroup with weak cancellation and $\omega\in\C_0(\SS)^*$ is an idempotent state with compact support (but cf.~Proposition \ref{ecpt} and Remark \ref{Rem_proper}) then there exits a proto-compact quantum hypergroup $(\sX_\omega,\Delta_\omega,\bh_\omega)$ defined as above and clearly we have
\[
\omega=\bh_\omega\comp\pi_\omega,
\]
where the map $\pi_\omega :\C_0(\SS)\ni{a}\mapsto{p_\omega^\perp}ap_\omega^\perp\in\sX_\omega$ is a completely positive and surjective. Moreover by Remark \ref{innywzor} we have
\begin{equation}\label{Deltapiomega}
\Delta_\omega\comp\pi_\omega=(\pi_\omega\tens\pi_\omega)\comp\Delta_\SS.
\end{equation}

One can therefore say that an idempotent state with compact support arises as the Haar state on a sub-proto-compact quantum hypergroup. As we already mentioned in the Introduction such a realization of $\omega$ is not unique. For example if $\sY=\CC$ with trivial proto-compact quantum hypergroup structure (in fact it is the trivial group) and $\pi_\sY=\omega:\C_0(\SS)\to\sY$ we also have $\omega=\bh_\sY\comp\pi_\sY$ with the obvious choice for $\bh_\sY$.

Quite obviously if $(\sY,\Delta_\sY,\bh_\sY)$ is a proto-compact quantum hypergroup and $\pi_\sY$ is a completely positive map $\C_0(\SS)\to\sY$ such that $(\pi_\sY\tens\pi_\sY)\comp\Delta_\SS=\Delta_\sY\comp\pi_\sY$ then $\omega=\bh_\sY\comp\pi_\sY$ is an idempotent state, so our considerations so far establish that all idempotent states on $\C_0(\SS)$ arise as Haar states on sub-proto-compact quantum hypergroups of $\SS$.

The role of the particular realization of $\omega$ through $\pi_\omega$ as described above is summarized in the next theorem.

\begin{theorem}\label{univ_pi_omega}
Let $\SS$ be a quantum semigroup with weak cancellation and let $\omega\in\C_0(\SS)^*$ be an idempotent state with compact support. Let $(\sY,\Delta_\sY,\bh_\sY)$ be a proto-compact quantum hypergroup and $\pi_\sY:\C_0(\SS)\to\sY$ a completely positive, completely contractive, surjective map such that $\omega=\bh_\sY\comp\pi_\sY$ and
\[
\Delta_\sY\comp\pi_\sY=(\pi_\sY\tens\pi_\sY)\comp\Delta_\SS.
\]
Then there exists a unique completely positive, completely contractive, surjective map $\pi_\sY^\omega:\sX_\omega\to\sY$ such that $\pi_\sY=\pi_\sY^\omega\comp\pi_\omega$ and
\begin{equation}\label{Delta3}
\Delta_\sY\comp\pi_\sY^\omega=(\pi_\sY^\omega\tens\pi_\sY^\omega)\comp\Delta_\omega.
\end{equation}
\end{theorem}

\begin{proof}
Let $a\in\ker{\pi_\omega}$. By \cite[Proposition 4.4]{Akemann-Pedersen-Tomiyama} we have $a=b+c$, where $b,c^*\in{C_\omega}$ and we recall that $C_\omega=\bigl\{x\in\C_0(\SS)\,\bigl|\bigr.\,\omega(x^*x)=0\bigr\}$. Then
\[
\bh_\sY\bigl(\pi_\sY(b^*b)\bigr)=\omega(b^*b)=0,
\]
and since $\bh_\sY$ is faithful, we obtain $\pi_\sY(b^*b)=0$. By Cauchy-Schwarz inequality we have $\pi_\sY(b)=0$. Similarly $\pi_\sY(c)=0$, and consequently $a\in\ker{\pi_\sY}$. Thus $\ker{\pi_\omega}\subset\ker{\pi_\sY}$ and hence there exists a linear surjective map $\pi^\omega_\sY:\sX_\omega\to\sY$ such that $\pi_\sY=\pi_\sY^\omega\comp\pi_\omega$. Since the isomorphism of $\sX_\omega$ and $\C_0(\SS)/\ker\pi_\omega$ is isometric (\cite[Proposition 4.4]{Akemann-Pedersen-Tomiyama}) the induced map $\pi_\sY^\omega$ is a contraction. Indeed, given $y \in\sX_\omega$ and $\varepsilon>0$ there exists $x\in\C_0(\SS)$ such $y = \pi_\omega(x)$ and $\|y\|+\varepsilon\geq \|x\|$. In particular \[\|\pi_\sY^\omega(y)\| = \|\pi_\sY(x)\|\leq\|x\|\leq \|y\|+\varepsilon\] and we see that $\pi^\omega_\sY$ is contractive.

For any $n>1$ the diagram
\[
\xymatrix{
M_n\tens\C_0(\SS)\ar[rr]^{\id\tens\pi_\omega}\ar[rd]_{\id\tens\pi_\sY}&&M_n\tens\sX_\omega\ar[ld]^{\id\tens\pi_\sY^\omega}\\
&M_n\tens\sY
}
\]
is commutative.\footnote{We have $\sum{e_{ij}\tens{a_{ij}}}\in\ker{\id\tens\pi_\omega}$ if and only if $a_{i,j}\in\ker{\pi_\omega}$ for all $i,j$, so $a_{i,j}\in\ker{\pi_\sY}$ for all $i,j$ which is equivalent to $\sum{e_{ij}\tens{a_{ij}}}\in\ker{\id\tens\pi_\sY}$, i.e.~$\ker{\id\tens\pi_\omega}\subset\ker{\id\tens\pi_\sY}$.} In order to see that $\id\tens\pi_\sY: M_n\tens\sX_\omega\to{M_n}\tens\sY $ is a contractive map for every $n>1$ it is enough to note that the left kernel of $\id\tens\pi_\omega:M_n\tens\C_0(\SS)\to{M_n}\tens\sX_\omega$ is generated by the projection $\I\tens{p_\omega^\perp}\in(M_n\tens\C_0(\SS))^{**}$ and then use \cite[Proposition 4.4]{Akemann-Pedersen-Tomiyama} the way it was used in the case $n=1$. Thus $\pi^\omega_\sY$ is a complete contraction.

In order to see that $\pi^\omega_\sY$ is completely positive we must show that it is unital and use \cite[Proposition 2.11]{Paulsen}. For this let $q=\pi^\omega_\sY(p_\omega^\perp)$. Now since $p_\omega^\perp$ is compact, there exists a positive $x\in\C_0(\SS)$ such that $p_\omega^\perp=\pi_\omega(x)$. Therefore $q$ which is also equal to $\pi_\sY(x)$ is positive and of norm less or equal to $1$. Now
\[
\bh_\sY(q)=\omega(x)=\bh_\omega(p_\omega^\perp)=1,
\]
so $\bh_\sY(\I_\sY-q)=0$. Hence $q=\I_\sY$ by faithfulness of $\bh_\sY$.

Finally to prove \eqref{Delta3} we apply $\pi^\omega_{\sY}\tens\pi^\omega_{\sY}$ to both sides of \eqref{Deltapiomega} obtaining
\[
(\pi_\sY^\omega\tens\pi_\sY^\omega)\comp\Delta_\omega\comp\pi_\omega=(\pi_\sY\tens\pi_\sY)\comp\Delta_\SS.
\]
But the right hand side is $\Delta_\sY\comp\pi_\sY=\Delta_\sY\comp\pi_\sY^\omega\comp\pi_\omega$, so we can cancel the surjective $\pi_\omega$ on both sides obtaining $(\pi_\sY^\omega\tens\pi_\sY^\omega)\comp\Delta_\omega=\Delta_\sY\comp\pi_\sY^\omega$.
\end{proof}

\section{Idempotent states of Haar type}\label{HaarType}

Let $\SS$ be a quantum semigroup and let $\omega\in\C_0(\SS)^*$ be an idempotent state. The results of Section \ref{kik} say that at least if $\SS$ has weak cancellation and $\omega$ has compact support (cf.~Proposition \ref{ecpt}) then $\omega$ factorizes through a proto-compact quantum hypergroup and this factorization may be chosen to have the universal property described in Theorem \ref{univ_pi_omega}. It is natural to distinguish those idempotent states which factor through a compact quantum subgroup of $\SS$:

\begin{definition}\label{Def_Haar_type}
Let $\SS$ be a quantum semigroup and let $\omega\in\C_0(\SS)^*$ be an idempotent state. We say that $\omega$ is \emph{of Haar type} if there is a compact quantum group $\KK$ with faithful Haar measure $\bh_\KK$ and a surjective $\pi_\KK\in\Mor(\C_0(\SS),\C(\KK))$ such that $(\pi_\KK\tens\pi_\KK)\comp\Delta_\SS=\Delta_\KK\comp\pi_\KK$ and $\omega=\bh_\KK\comp\pi_\KK$.
\end{definition}

\begin{remark}
Let $\sA$ be a \cst-algebra, $\sB$ a unital $\cst$-algebra and $\pi\in\Mor(\sA,\sB)$ a surjective morphism. The central carrier $p\in\sA^{**}$ of $\pi$ is easily checked to be compact. Assuming that $\sA=\C_0(\SS)$, $\sB=\C(\KK)$ and $\omega=\bh_\KK\comp\pi_\KK$ as considered in Definition \ref{Def_Haar_type} we conclude that $p_\omega^\perp$ is a compact projection. Moreover, in this case $p_\omega^\perp$ is central.
\end{remark}

\begin{theorem}\label{equiv}
Let $\SS$ be a quantum semigroup with weak cancellation and let $\omega\in\C_0(\SS)^*$ be an idempotent state with compact support. Then the following conditions are equivalent:
\begin{enumerate}
\item\label{equiv1} the ideal $C_\omega$ is two sided,
\item\label{equiv2} $p_\omega^\perp$ is a central projection,
\item\label{equiv2.5} $\sX_\omega$ is a \cst-algebra and $\pi_\omega$ is a $*$-homomorphism,
\item\label{equiv3} $(\sX_\omega,\Delta_\omega)$ describes a compact quantum group,
\item\label{equiv4} $\omega$ is of Haar type.
\end{enumerate}
\end{theorem}

\begin{proof}
If $C_\omega$ is a two sided ideal then $p_\omega$ and hence $p_\omega^\perp$ is central (\cite[Section 3.20]{Stratila-Zsido}). Thus \eqref{equiv1} implies \eqref{equiv2}.

Now if $p_\omega^\perp$ is central, the operator system $\sX_\omega$ is actually a unital \cst-algebra and $\pi_\omega$ is a $*$-homomorphism because it is the composition of a $*$-homomorphism $\tDelta$ and the map $x\tens{y}\mapsto{xp_\omega^\perp}\tens{yp_\omega^\perp}$ on $\C_0(\SS)^{**}\operatorname{\bar{\tens}}\C_0(\SS)^{**}$ which is also a $*$-homomorphism. This means that \eqref{equiv2} implies \eqref{equiv2.5}.

Assume \eqref{equiv2.5}. Then for $x,y\in\sX_\omega$ we can choose $a,b\in\C_0(\SS)$ such that $x=\pi_\omega(a)$ and $y=\pi_\omega(b)$. Moreover we have $xy=\pi_\omega(ab)$ and
\begin{align*}
\Delta_\omega(xy)=\bigl((\pi_\omega\tens\pi_\omega)\comp\Delta_\SS\bigr)(ab)&=\bigl((\pi_\omega\tens\pi_\omega)\comp\Delta_\SS\bigr)(a)\bigl((\pi_\omega\tens\pi_\omega)\comp\Delta_\SS\bigr)(b)\\
&=(\Delta_\omega\comp\pi_\omega)(a)(\Delta_\omega\comp\pi_\omega)(b)=\Delta_\omega(x)\Delta_\omega(y),
\end{align*}
so that $\Delta_\omega$ is a $*$-homomorphism which is moreover unital. $(\sX_\omega,\Delta_\omega)$ defines a compact quantum semigroup $\KK$ by $\C(\KK)=\sX_\omega$, and since $\pi_\omega$ is a surjective $*$-homomorphism intertwining comultiplications, the weak cancellation laws of $\SS$ carry over to weak cancellation laws for $\KK$. Thus by \cite[Theorem 3.2]{Murphy-Tuset} $\KK$ is a compact quantum group whose Haar measure is faithful by Theorem \ref{Xomega}. This shows that \eqref{equiv2.5} implies \eqref{equiv3}.

\eqref{equiv3} and \eqref{equiv4} are equivalent by definition of an idempotent state of Haar type and finally if \eqref{equiv4} holds then $\omega=\bh_\KK\comp\pi_\KK$, where $\pi_\KK$ is a $*$-homomorphism from $\C_0(\SS)$ onto $\C(\KK)$, where $\KK$ is a compact quantum group with faithful Haar measure $\bh_\KK$. Thus
\[
C_\omega=\bigl\{b\in\C_0(\SS)\,\bigl|\bigr.\,\omega(b^*b)=0\bigr\}=\bigl\{b\in\C_0(\SS)\,\bigl|\bigr.\,\bh_\KK\bigl(\pi_\KK(b)^*\pi_\KK(b)\bigr)=0\bigr\}=\ker\pi_\omega
\]
by faithfulness of $\bh_\KK$. Hence $C_\omega$ is a two sided ideal.
\end{proof}

\begin{corollary}\label{trace}
Let $\SS$ be a quantum semigroup with weak cancellation and let $\omega\in\C_0(\SS)^*$ be a tracial idempotent state with compact support. Then $\omega$ is of Haar type.
\end{corollary}

\begin{proof}
If $\omega$ is a trace, clearly condition \eqref{equiv1} of Theorem \ref{equiv} is satisfied.
\end{proof}

\begin{remark}
Suppose that $\SS$ is a locally compact quantum semigroup with weak cancellation and $\omega\in\C_0(\SS)^*$ an idempotent state which is central. We could not prove that in this case $p_\omega^\perp$ is compact which then would imply that $\omega$ is of Haar type. However if $\SS$ is a classical semigroup this follows from \cite{Mukherjea_Tserpes} as explained in Section \ref{Intro}.
\end{remark}

\section{Further questions and comments}

In what follows we adopt the convention and notation concerning idempotent state $\omega\in\C_0(\SS)^*$ on a locally compact quantum semigroup $\SS$ satisfying weak cancellation law. Our study leads to the following questions:
\begin{enumerate}
\item Is $p_\omega^\perp\in\C_0(\SS)^{**}$ compact?
\item Is there an idempotent state $\omega$ such that $\sX_\omega$ is not a \cst-subalgebra of $p_\omega^\perp\C_0(\SS)^{**}p_\omega^\perp$? An example of a \cst-algebra $\sA$ and a projection $p\in\sA^{**}$ such that $p\sA{p}$ is not a $\cst$-subalgebra of $p\sA^{**}p$ was given in \cite[Theorem 4.5]{Akemann-Pedersen-Tomiyama}. As was proved in \cite[Theorem 3.1]{Brown}, $p\sA{p}$ does form a \cst-subalgebra of $p\sA^{**}p$ if and only if $p$ satisfies so called MSQC-condition: for every self-adjoint $h\in{p}\sA{p}$ there exists a self adjoint element $a\in\sA$ such that $h=ap=pa$.
\end{enumerate}

\subsection*{Acknowledgments}

The authors wish to thank Massoud Amini who greatly contributed to an early version of this paper and Adam Skalski and Haonan Zhang for fruitful discussions.


\begin{thebibliography}{10}

\bibitem{Akemann}
C.~A. Akemann.
\newblock A {G}elfand representation theory for {$C^{\ast}$}-algebras.
\newblock {\em Pacific J. Math.}, 39:1--11, 1971.

\bibitem{Akemann-Pedersen-Tomiyama}
C.~A. Akemann, G.~K. Pedersen, and J.~Tomiyama.
\newblock Multipliers of {$C\sp*$}-algebras.
\newblock {\em J. Functional Analysis}, 13:277--301, 1973.

\bibitem{Baaj-Blanchard-Skandalis}
S.~Baaj, E.~Blanchard, and G.~Skandalis.
\newblock Unitaires multiplicatifs en dimension finie et leurs sous-objets.
\newblock {\em Ann. Inst. Fourier (Grenoble)}, 49(4):1305--1344, 1999.

\bibitem{Baaj-Skandalis}
S.~Baaj and G.~Skandalis.
\newblock Unitaires multiplicatifs et dualit\'{e} pour les produits crois\'{e}s
  de {$C^*$}-alg\`ebres.
\newblock {\em Ann. Sci. \'{E}cole Norm. Sup. (4)}, 26(4):425--488, 1993.

\bibitem{Blecher-Neal}
D.~P. Blecher and M.~Neal.
\newblock Noncommutative topology and jordan operator algebras.
\newblock {\em Math. Nachr.}, 292(3):481--510, 2018.

\bibitem{Brown}
L.~G. Brown.
\newblock M{ASA}'s and certain type {I} closed faces of {$C^*$}-algebras.
\newblock In {\em Group representations, ergodic theory, and mathematical
  physics: a tribute to {G}eorge {W}. {M}ackey}, volume 449 of {\em Contemp.
  Math.}, pages 69--98. Amer. Math. Soc., Providence, RI, 2008.

\bibitem{Chapovsky-Vainerman}
Yu.~A. Chapovsky and L.~I. Vainerman.
\newblock Compact quantum hypergroups.
\newblock {\em J. Operator Theory}, 41(2):261--289, 1999.

\bibitem{Cohen}
P.~J. Cohen.
\newblock On a conjecture of {L}ittlewood and idempotent measures.
\newblock {\em Amer. J. Math.}, 82:191--212, 1960.

\bibitem{DKSS}
M.~Daws, P.~Kasprzak, A.~Skalski, and P.~M. So\l{}tan.
\newblock Closed quantum subgroups of locally compact quantum groups.
\newblock {\em Adv. Math.}, 231(6):3473--3501, 2012.

\bibitem{Franz-Skalski_new}
U.~Franz and A.~Skalski.
\newblock A new characterisation of idempotent states on finite and compact
  quantum groups.
\newblock {\em C. R. Math. Acad. Sci. Paris}, 347(17-18):991--996, 2009.

\bibitem{Franz-Skalski_new_e}
U.~Franz and A.~Skalski.
\newblock A new characterisation of idempotent states on finite and compact
  quantum groups.
\newblock {\em arXiv e-prints}, page arXiv:0906.2362, Jun 2009.

\bibitem{Franz-Skalski2009}
U.~Franz and A.~Skalski.
\newblock On idempotent states on quantum groups.
\newblock {\em J. Algebra}, 322(5):1774--1802, 2009.

\bibitem{Heyer}
H.~Heyer.
\newblock {\em Probability measures on locally compact groups}.
\newblock Springer-Verlag, Berlin-New York, 1977.
\newblock Ergebnisse der Mathematik und ihrer Grenzgebiete, Band 94.

\bibitem{Kac-Paljutkin}
G.~I. Kac and V.~G. Paljutkin.
\newblock Finite ring groups.
\newblock {\em Trudy Moskov. Mat. Ob\v{s}\v{c}.}, 15:224--261, 1966.

\bibitem{lattice}
P.~Kasprzak and P.~M. So{\l}tan.
\newblock The lattice of idempotent states on a locally compact quantum group.
\newblock {\em arXiv e-prints, To appear in Publ. Res. Inst. Math. Sci.}, page
  arXiv:1802.03953, Feb 2018.

\bibitem{Kawada-Ito}
Y.~Kawada and K.~It\^{o}.
\newblock On the probability distribution on a compact group. {I}.
\newblock {\em Proc. Phys.-Math. Soc. Japan (3)}, 22:977--998, 1940.

\bibitem{Kelley}
J.~L. Kelley.
\newblock Averaging operators on {$C_{\infty }(X)$}.
\newblock {\em Illinois J. Math.}, 2:214--223, 1958.

\bibitem{Kustermans}
J.~Kustermans.
\newblock Locally compact quantum groups in the universal setting.
\newblock {\em Internat. J. Math.}, 12(3):289--338, 2001.

\bibitem{Lance}
E.~C. Lance.
\newblock {\em Hilbert {$C^*$}-modules}, volume 210 of {\em London Mathematical
  Society Lecture Note Series}.
\newblock Cambridge University Press, Cambridge, 1995.
\newblock A toolkit for operator algebraists.

\bibitem{Masuda-Nakagami-Woronowicz}
T.~Masuda, Y.~Nakagami, and S.~L. Woronowicz.
\newblock A {$C^\ast$}-algebraic framework for quantum groups.
\newblock {\em Internat. J. Math.}, 14(9):903--1001, 2003.

\bibitem{Mukherjea_Tserpes}
A.~Mukherjea and N.~A. Tserpes.
\newblock Idempotent measures on locally compact semigroups.
\newblock {\em Proc. Amer. Math. Soc.}, 29:143--150, 1971.

\bibitem{Murphy-Tuset}
G.~J. Murphy and L.~Tuset.
\newblock Aspects of compact quantum group theory.
\newblock {\em Proc. Amer. Math. Soc.}, 132(10):3055--3067, 2004.

\bibitem{Pal}
A.~Pal.
\newblock A counterexample on idempotent states on a compact quantum group.
\newblock {\em Lett. Math. Phys.}, 37(1):75--77, 1996.

\bibitem{Paulsen}
V.~Paulsen.
\newblock {\em Completely bounded maps and operator algebras}, volume~78 of
  {\em Cambridge Studies in Advanced Mathematics}.
\newblock Cambridge University Press, Cambridge, 2002.

\bibitem{Pedersen}
G.~K. Pedersen.
\newblock {\em {$C^{\ast} $}-algebras and their automorphism groups}, volume~14
  of {\em London Mathematical Society Monographs}.
\newblock Academic Press, Inc. [Harcourt Brace Jovanovich, Publishers],
  London-New York, 1979.

\bibitem{pym1962}
J.~S. Pym.
\newblock Idempotent measures on semigroups.
\newblock {\em Pacific J. Math.}, 12:685--698, 1962.

\bibitem{Rudin}
W.~Rudin.
\newblock Idempotents in group algebras.
\newblock {\em Bull. Amer. Math. Soc.}, 69:224--227, 1963.

\bibitem{Salmi_strict}
P.~Salmi.
\newblock {Subgroups and strictly closed invariant C*-subalgebras}.
\newblock {\em arXiv e-prints}, page arXiv:1110.5459, Oct 2011.

\bibitem{Salmi-Skalski}
P.~Salmi and A.~Skalski.
\newblock Idempotent states on locally compact quantum groups.
\newblock {\em Q. J. Math.}, 63(4):1009--1032, 2012.

\bibitem{mmu}
P.~M. So\l{}tan and S.~L. Woronowicz.
\newblock From multiplicative unitaries to quantum groups. {II}.
\newblock {\em J. Funct. Anal.}, 252(1):42--67, 2007.

\bibitem{Stratila-Zsido}
\c{S}. Str\u{a}til\u{a} and L.~Zsid\'{o}.
\newblock {\em Lectures on von {N}eumann algebras}.
\newblock Editura Academiei, Bucharest; Abacus Press, Tunbridge Wells, 1979.
\newblock Revision of the 1975 original, Translated from the Romanian by Silviu
  Teleman.

\bibitem{Sun-Tserpes}
T.~Sun and N.~A. Tserpes.
\newblock Idempotent measures on locally compact semigroups.
\newblock {\em Z. Wahrscheinlichkeitstheorie und Verw. Gebiete}, 15:273--278,
  1970.

\bibitem{Taylor}
D.~C. Taylor.
\newblock The strict topology for double centralizer algebras.
\newblock {\em Trans. Amer. Math. Soc.}, 150:633--643, 1970.

\bibitem{unbo}
S.~L. Woronowicz.
\newblock Unbounded elements affiliated with {$C^*$}-algebras and noncompact
  quantum groups.
\newblock {\em Comm. Math. Phys.}, 136(2):399--432, 1991.

\end{thebibliography}
\end{document}